\documentclass[pdflatex,sn-mathphys-num]{sn-jnl}

\usepackage{graphicx}%
\usepackage{multirow}%
\usepackage{amsmath,amssymb,amsfonts}%
\usepackage{amsthm}%
\usepackage{mathrsfs}%
\usepackage[title]{appendix}%
\usepackage{xcolor}%
\usepackage{textcomp}%
\usepackage{manyfoot}%
\usepackage{booktabs}%
\usepackage{algorithm}%
\usepackage{algorithmicx}%
\usepackage{algpseudocode}%
\usepackage{listings}%
\usepackage{url}
\theoremstyle{thmstyleone}%
\theoremstyle{thmstyletwo}%
\theoremstyle{thmstylethree}%
 \newtheorem{thm}{Theorem}[section]

 \theoremstyle{definition}
 
 \theoremstyle{remark}

 \numberwithin{equation}{section}
\begin{document}
\title[Equivalence of the minimality Conditions ]{Equivalence of the minimality conditions for the root functions of Sturm-Liouville problems with a boundary condition depending linearly on an eigenparameter}
\author[1]{\fnm{Yagub} \sur{Aliyev}}\email{yaliyev@ada.edu.az}

\author*[2]{\fnm{Narmin} \sur{Aliyeva}}\email{narmin.aliyeva@eas.az}
\equalcont{These authors contributed equally to this work.}

\affil*[1]{\orgdiv{School of IT and Engineering}, \orgname{ADA University}, \orgaddress{\street{61 Ahmadbay Agha-Oglu Street}, \city{Baku }, \postcode{AZ1008}, \country{Azerbaijan}}}

\affil[2]{\orgdiv{Department of Differential Equations and Control Theory}, \orgname{Baku State University}, \orgaddress{\street{Academic Zahid Khalilov str. 23}, \city{Baku}, \postcode{AZ1148}, \country{Azerbaijan}}}


\abstract{We study the minimality of the system of root functions associated with a Sturm--Liouville problem whose boundary condition depends linearly on the eigenparameter. Two different criteria for minimality were previously obtained using independent approaches. In this paper, we establish the equivalence of these criteria and provide a unified characterization of the exceptional cases in which the removal of certain associated functions fails to preserve minimality. The theoretical results are illustrated by several examples involving multiple eigenvalues, demonstrating the consistency of the two approaches and clarifying the structure of the corresponding root function systems.}
\keywords{Sturm–Liouville problem, eigenparameter-dependent boundary conditions, characteristic function, root functions, associated functions, biorthogonal system, minimality.}
\pacs[MSC Classification]{34B24; 34L10.}

\maketitle
\section{Introduction}

We consider the following Sturm-Liouville problem
\begin{equation}
  -y^{\prime \prime }+q(x)y=\lambda y,\ 0<x<1,  \label{(1.1)} 
\end{equation}
\begin{equation}
y(0)\cos \beta =y^{\prime }(0)\sin \beta , 
\label{(1.2)}
\end{equation}
\begin{equation}
(a\lambda + b)y(1) = (c\lambda + d)y{'}(1). \label{(1.3)}
\end{equation}
Here $a,b,c,d$ are reals, $\ 0\leq \beta <\pi$  and $ad-bc<0$. We assume that $q(x)$ is a real-valued continuous function on the interval $[0,1]$. In our previous papers \cite{aliyev5} and \cite{aliyev6} we developed two different approaches for the minimality and basis properties of the root functions of this problem. In the current paper, we will compare  these two approaches and provide some new examples.  

Problems similar to (\ref{(1.1)})--(\ref{(1.3)}) were  studied in \cite{fulton, aliyev1, zaliyev1, kerimov, mois}. A more general polynomial dependence on $\lambda$ was investigated in \cite{rus} and \cite{shkalikov1}. In a recent paper \cite{shkalikov2}, Theorem 3, it was mentioned that for some choices of the removed associated functions the system of root functions sometimes does not form a minimal system. 
It was further shown in \cite{shkalikov2} that, for an eigenvalue $\lambda_0$ of multiplicity $2$, the system of root functions obtained by removing the eigenfunction $y_0$ is not necessarily minimal in $L_2(0,1)$, whereas the system obtained by removing the associated function $y_1$ is always minimal. Moreover, if $\lambda_0$ is an eigenvalue of multiplicity $3$, then the systems obtained by removing either the eigenfunction $y_0$ or the associated function $y_1$ are not necessarily minimal in $L_2(0,1)$, whereas the system obtained by removing the associated function $y_2$ is always minimal. The geometric intuition for why some choices of the associated functions fail the minimality is given in comments preceding Theorem 3 of \cite{shkalikov2}. In \cite{aliyev5} the necessary and sufficient conditions for minimality were written in the forms $\mathfrak{A}(y_{1}^{*}) = 0$, $\mathfrak{A}(y_{1}^{\#}) = 0$, and $\mathfrak{A}(y_{2}^{\#}) = 0$.
Furthermore, in \cite{shkalikov2}, it was mentioned that these cases are exceptional. Note that in \cite{shkalikov2}, instead of \eqref{(1.2)}, a more specialized first boundary condition $y'(0)=0$ was assumed. In \cite{aliyev6}, these exceptional cases were identified by the equalities $C=-\frac{\omega^{\prime\prime\prime}(\lambda_0)}{3\omega^{\prime\prime}(\lambda_0)}$, $C=-\frac{\omega^{IV}(\lambda_0)}{4\omega^{\prime\prime\prime}(\lambda_0)}$, and 
$$
 D=C^2+
\frac{\omega^{IV}(\lambda_0)}{4\omega^{\prime\prime\prime}(\lambda_0)}\left(C+\frac{\omega^{IV}(\lambda_0)}{4\omega^{\prime\prime\prime}(\lambda_0)}\right)-\frac{\omega^{V}(\lambda_0)}{20\omega^{\prime\prime\prime}(\lambda_0)}. 
$$
In the current paper, we will prove the equivalence of the approaches in \cite{aliyev5} and \cite{aliyev6}, and demonstrate these exceptional cases using concrete examples.

\section{Notations and preliminary results.}
Suppose that $y(x,\lambda)$ is the solution of the boundary value problem  
\begin{equation}
-y^{\prime \prime }+q(x)y=\lambda y,  \label{(2.1)}
\end{equation}
\begin{equation}
y(0)=\sin \beta ,\  y^{\prime }(0)=\cos \beta .   \label{(2.2)}
\end{equation}
The characteristic function is defined as 
\begin{equation}
\omega(\lambda) = (a\lambda + b)y(1,\lambda) - (c\lambda + d)y{'}(1,\lambda).\label{(2.3)}
\end{equation}
Suppose that  $\lambda _{0}$ is a multiple eigenvalue.
The eigenfunction $y_{0}=y(x,\lambda_0)$ satisfy
\begin{equation}
-y_{0}^{\prime \prime }+q(x)y_{0}=\lambda _{0}y_{0},  \label{(2.4)}
\end{equation}
\begin{equation}
y_{0}^{\prime }(0)\sin \beta =y_{0}(0)\cos \beta ,  \label{(2.5)}
\end{equation}
\begin{equation}
(a\lambda _{0}+b)y_{0}(1)=(c\lambda _{0}+d)y^{\prime }_{0}(1).  \label{(2.6)}
\end{equation}
The following notation will be used throughout the paper:
\begin{equation}
\mathfrak{A}(y_0) =
\begin{cases}
\displaystyle  
\dfrac{y_0(1)}{c\lambda_0 + d}, 
& \text{if } \lambda_0 \neq -\dfrac{d}{c}, \\[10pt]
\displaystyle \ 
\dfrac{y_0'(1)}{a\lambda_0 + b}, 
& \text{if } \lambda_0 = -\dfrac{d}{c}.
\end{cases}\label{(2.11)}
\end{equation}
Note that if $\lambda _{0}$ is a double or triple eigenvalue, then 
\begin{equation}
(y_{0},y_0)=-(ad-bc)\mathfrak{A}^2(y_{0}).
\label{(2.17)}
\end{equation}
The first associated function
\(y_{1}\) of the eigenfunction \(y_{0}\) is defined by (see \cite{naim})
\begin{equation}
    {- y''_{1}} + q(x)y_{1} = \lambda_{0}y_{1} + y_{0}, \label{(3.1)}
\end{equation}
\begin{equation}
y_{1}(0)\cos\beta = y'_{ 1}(0)\sin\beta, \label{(3.2)}
\end{equation}
\begin{equation}
\left( a\lambda_{0} + b \right)y_{1}(1) + ay_{0}(1) = \left( c\lambda_{0} + d \right)y'_{1}(1) + cy'_{0}(1). \label{(3.3)}
\end{equation}
Let
$$\tilde{y}_{1}=\lim_{\lambda\to\lambda_0}y_\lambda(x,\lambda),$$and $y_{1}=\tilde{y}_{1}+Cy_0$ for some constant $C$.
Let us define 
\begin{equation}
\mathfrak{A}(y_{1}) =
\begin{cases}
\displaystyle  
\frac{y_{1}(1)}{c\lambda_{0}+d} - \frac{cy_{0}(1)}{{(c\lambda_{0}+d)}^{2}}, 
& \text{if } \lambda_0 \neq -\dfrac{d}{c}, \\[10pt]
\displaystyle \ 
\frac{y'_{1}(1)}{a\lambda_{0}+b} - \frac{ay'_{0}(1)}{{(a\lambda_{0}+b)}^{2}}, 
& \text{if } \lambda_0 = -\dfrac{d}{c}.
\end{cases}\label{(3.7)}
\end{equation}

\begin{equation}
\mathfrak{A}(\tilde{y}_{1}) =
\begin{cases}
\displaystyle  
\frac{\tilde{y}_{ 1}(1)}{c\lambda_{0}+d} - \frac{cy_{0}(1)}{{(c\lambda_{0}+d)}^{2}}, 
& \text{if } \lambda_0 \neq -\dfrac{d}{c}, \\[10pt]
\displaystyle \ 
\frac{\tilde{y}'_{1}(1)}{a\lambda_{0}+b} - \frac{ay'_{0}(1)}{{(a\lambda_{0}+b)}^{2}}, 
& \text{if } \lambda_0 = -\dfrac{d}{c}.
\end{cases}\label{(3.7.9)}
\end{equation}
Denote
\begin{equation}
   T_0:=(y_{1},y_{0})+(ad-bc)\mathfrak{A}(y_{1})\cdot \mathfrak{A}(y_{0}). \label{(Tk)} 
\end{equation}
Note that if  $\lambda _{0}$ is a multiple eigenvalue, 
then 
\begin{equation}
T_0=\mathfrak{A}(y_0)\cdot\frac{\omega^{\prime\prime}(\lambda_0)}{2}. \label{(3.16)}
\end{equation}

\noindent Similarly, for 
$\hat{y}_{1}=y_1+Cy_k$ define
\begin{equation}
\mathfrak{A}(\hat{y}_{1}) =
\begin{cases}
\displaystyle  
\frac{\hat{y}_{ 1}(1)}{c\lambda_{0}+d} - \frac{cy_{0}(1)}{{(c\lambda_{0}+d)}^{2}}, 
& \text{if } \lambda_0 \neq -\dfrac{d}{c}, \\[10pt]
\displaystyle \ 
\frac{\hat{y}'_{1}(1)}{a\lambda_{0}+b} - \frac{ay'_{0}(1)}{{(a\lambda_{0}+b)}^{2}}, 
& \text{if } \lambda_0 = -\dfrac{d}{c}.
\end{cases}\label{(3.16.1)}
\end{equation}
Suppose that  $\lambda _{0}$ is a multiple eigenvalue.
Denote
\begin{equation}
Q_0:=(y_{1},y_1)+(ad-bc)\mathfrak{A}^2(y_{1}).
\label{(2.17.1)}
\end{equation}
Note that
\begin{equation}
Q_0=\mathfrak{A}(\hat{y}_{1})\frac{\omega^{\prime\prime}(\lambda_0)}{2}+\mathfrak{A}(y_0)\cdot\frac{\omega^{\prime\prime\prime}(\lambda_0)}{6}. \label{(3.16.9)}
\end{equation}
If \(\lambda_{0}\) is a triple  eigenvalue, then the second associated function
\(y_{2}\) of the first associated function \(y_{1}\) is defined by (see \cite{naim})
\begin{equation}
 -y_{2}^{\prime \prime }+q(x)y_{2}=\lambda _{0}y_{2}+y_{1}, \label{(4.1)}  
\end{equation}
\begin{equation}
   y_{2}(0)\cos \beta  =y_{2}^{\prime }(0)\sin \beta,     \label{(4.2)}
\end{equation}
\begin{equation}
 \left( a\lambda_{0} + b \right)y_{ 2}(1) + ay_{1}(1) = \left( c\lambda_{0} + d \right)y'_{2}(1) + cy'_{1}(1).  \label{(4.3)}   \end{equation}
Let
$$\tilde{y}_{2}=\frac{1}{2}\lim_{\lambda\to\lambda_0}y_{\lambda\lambda}(x,\lambda),$$and $y_{2}=\tilde{y}_{2}+C\tilde{y}_{1}+Dy_0$ for some constant $D$.
We also define 
\begin{equation}
\mathfrak{A}(y_{2}) =
\begin{cases}
\displaystyle  
\frac{y_{2}(1)}{c\lambda_{0} + d} - \frac{cy_{1}(1)}{\left( c\lambda_{0} + d \right)^{2}}
+\frac{c^2y_{0}(1)}{\left( c\lambda_{0} + d \right)^{3}}, 
&  \text{if }  \lambda_0 \neq -\dfrac{d}{c}, \\[10pt]
\displaystyle \ 
\frac{y'_{2}(1)}{a\lambda_{0} + b} - \frac{ay'_{1}(1)}{\left( a\lambda_{0} + b \right)^{2}}
+\frac{a^2y'_{0}(1)}{\left( a\lambda_{0} + b\right)^{3}}, 
& \text{if } \lambda_0 = -\dfrac{d}{c},
\end{cases}\label{(4.7)}
\end{equation}
\begin{equation}
\mathfrak{A}(\tilde{y}_{2}) =
\begin{cases}
\displaystyle  
\frac{\tilde{y}_{2}(1)}{c\lambda_{0} + d} - \frac{c\tilde{y}_{1}(1)}{\left( c\lambda_{0} + d \right)^{2}}
+\frac{c^2y_{0}(1)}{\left( c\lambda_{0} + d \right)^{3}}, 
&  \text{if }  \lambda_0 \neq -\dfrac{d}{c}, \\[10pt]
\displaystyle \ 
\frac{\tilde{y}'_{2}(1)}{a\lambda_{0} + b} - \frac{a\tilde{y}'_{1}(1)}{\left( a\lambda_{0} + b \right)^{2}}
+\frac{a^2y'_{0}(1)}{\left( a\lambda_{0} + b\right)^{3}}, 
& \text{if } \lambda_0 = -\dfrac{d}{c}.
\end{cases}\label{(4.8)}
\end{equation}
Note that in the triple eigenvalue case $T_0=0$ and
\begin{equation}
   Q_0=(y_{2},y_{0})+(ad-bc)\mathfrak{A}(y_{2})\cdot \mathfrak{A}(y_{0})=\mathfrak{A}(y_0)\cdot\frac{\omega^{\prime\prime\prime}(\lambda_0)}{6}\neq 0. \label{(Qk)} 
\end{equation}
If $\lambda_0$ is an eigenvalue of multiplicity two, then define 
\(
y_{1}^* = y_{1} + C_1 y_0,
\)
where
\begin{equation}
    C_1=-\frac{\mathfrak{A}(\hat{y}_{1})\frac{\omega^{\prime\prime}(\lambda_0)}{2}+\mathfrak{A}(y_0)\frac{\omega^{\prime\prime\prime}(\lambda_0)}{6}}{\mathfrak{A}(y_0)\frac{\omega^{\prime\prime}(\lambda_0)}{2}}. \label{(5.1)}
\end{equation}
Denote 
\begin{equation}
 \mathfrak{A}(y^{*}_{1}) =
\begin{cases}
\displaystyle  
\frac{y^{*}_{1}(1)}{c\lambda_{0}+d} - \frac{cy_{0}(1)}{{(c\lambda_{0}+d)}^{2}}, 
& \text{if } \lambda_0 \neq -\dfrac{d}{c}, \\[10pt]
\displaystyle \ 
\frac{(y^{*}_{1})'(1)}{a\lambda_{0}+b} - \frac{ay'_{0}(1)}{{(a\lambda_{0}+b)}^{2}}, 
&  \text{if } \lambda_0 = -\dfrac{d}{c}.
\end{cases} \label{(5.2.1)}
\end{equation}
Note that
$C_1 = -\frac{Q_0}{T_0}$,

If $\lambda_0$ is a triple eigenvalue, then define 
\(
y_{1}^\# = y_{1} + C_2 y_0,
\)
where
\begin{equation}
    C_2=-\frac{\mathfrak{A}(\hat{y}_{1})\frac{\omega^{\prime\prime\prime}(\lambda_0)}{6}+\mathfrak{A}(y_0)\frac{\omega^{IV}(\lambda_0)}{24}}{\mathfrak{A}(y_0)\frac{\omega^{\prime\prime\prime}(\lambda_0)}{6}}. \label{(5.3)}
\end{equation}
Denote
\begin{equation}
 \mathfrak{A}(y^{\#}_{1}) =
\begin{cases}
\displaystyle  
\frac{y^{\#}_{1}(1)}{c\lambda_{0}+d} - \frac{cy_{0}(1)}{{(c\lambda_{0}+d)}^{2}}, 
& \text{if } \lambda_0 \neq -\dfrac{d}{c}, \\[10pt]
\displaystyle \ 
\frac{(y^{\#}_{1})'(1)}{a\lambda_{0}+b} - \frac{ay'_{0}(1)}{{(a\lambda_{0}+b)}^{2}}, 
& \text{if } \lambda_0 = -\dfrac{d}{c}.
\end{cases} \label{(5.5.1)}  
\end{equation}
\begin{equation}
   L_{0}=(y_{1},y_{2})+(ad-bc) \mathfrak{A}(y_{1})
 \mathfrak{A}(y_{2}), \label{(2.28)} 
\end{equation}
Note that $C_2=-\frac{L_0}{Q_0}$.
For the function \(y^{*}_{2}=y_{2}+C_{2}y_{1}\), denote
\begin{equation}
   \mathfrak{A}(y^{*}_{2}) =
\begin{cases}
\displaystyle  
\frac{y^{*}_{2}(1)}{c\lambda_{0} + d} - \frac{cy^{\#}_{1}(1)}{\left( c\lambda_{0} + d \right)^{2}}
+\frac{c^2y_{0}(1)}{\left( c\lambda_{0} + d \right)^{3}}, 
& \text{if } \lambda_0 \neq -\dfrac{d}{c}, \\[10pt]
\displaystyle \ 
\frac{(y^{*}_{2})'(1)}{a\lambda_{0} + b} - \frac{a(y^{\#}_{1})'(1)}{\left( a\lambda_{0} + b \right)^{2}}
+\frac{a^2y'_{0}(1)}{\left( a\lambda_{0} + b\right)^{3}}, 
& \text{if } \lambda_0 = -\dfrac{d}{c}.
\end{cases} \label{(5.5.2)}
\end{equation}
Denote also 
\begin{equation}
J_0=(y_{2}^*, y_{2})+(ad-bc)\mathfrak{A}(y_{2}^*)\mathfrak{A}(y_{2}). \label{(5.7)}
\end{equation}
Note that 
 \begin{equation}
   J_0 
    ={\mathfrak{A}(\hat{y}_{2})\frac{\omega^{\prime\prime\prime}(\lambda_0)}{6}+\mathfrak{A}(\hat{y}_{1})\frac{\omega^{IV}(\lambda_0)}{24}}+{\mathfrak{A}(y_0)\frac{\omega^{V}(\lambda_0)}{120}}-C_2^2\mathfrak{A}(y_0)\frac{\omega^{\prime\prime\prime}(\lambda_0)}{6}, \label{(5.8)}
\end{equation}
where 
$\hat{y}_{2}=y_2+Cy_1+Dy_0$ and
\begin{equation}
\mathfrak{A}(\hat{y}_{2}) =
\begin{cases}
\displaystyle  
\frac{\hat{y}_{2}(1)}{c\lambda_{0} + d} - \frac{c\hat{y}_{1}(1)}{\left( c\lambda_{0} + d \right)^{2}}
+\frac{c^2y_{0}(1)}{\left( c\lambda_{0} + d \right)^{3}}, 
&  \text{if }  \lambda_0 \neq -\dfrac{d}{c}, \\[10pt]
\displaystyle \ 
\frac{\hat{y}'_{2}(1)}{a\lambda_{0} + b} - \frac{a\hat{y}'_{1}(1)}{\left( a\lambda_{0} + b \right)^{2}}
+\frac{a^2y'_{0}(1)}{\left( a\lambda_{0} + b\right)^{3}}, 
& \text{if } \lambda_0 = -\dfrac{d}{c}.
\end{cases}\label{(5.5.4)}
\end{equation}

If $\lambda_0$ is a triple eigenvalue, then define
$y^{\#}_{2}=y^{*}_{2}+D_{1}y_{0}$,
where
\begin{equation}
    D_1 =-\frac{{\mathfrak{A}(\hat{y}_{2})\frac{\omega^{\prime\prime\prime}(\lambda_0)}{6}+\mathfrak{A}(\hat{y}_{1})\frac{\omega^{IV}(\lambda_0)}{24}}+{\mathfrak{A}(y_0)\frac{\omega^{V}(\lambda_0)}{120}}-C_2^2\mathfrak{A}(y_0)\frac{\omega^{\prime\prime\prime}(\lambda_0)}{6}}{\mathfrak{A}(y_0)\frac{\omega^{\prime\prime\prime}(\lambda_0)}{6}}. \label{(5.9)}
\end{equation}
Note that $D_1=-\frac{J_0}{Q_0}$. Denote also
\begin{equation}
\mathfrak{A}(y^{\#}_{2}) =
\begin{cases}
\displaystyle  
\frac{y^{\#}_{2}(1)}{c\lambda_{0} + d} - \frac{cy^{\#}_{1}(1)}{\left( c\lambda_{0} + d \right)^{2}}
+\frac{c^2y_{0}(1)}{\left( c\lambda_{0} + d \right)^{3}}, 
& \text{if } \lambda_0 \neq -\dfrac{d}{c}, \\[10pt]
\displaystyle \ 
\frac{(y^{\#}_{2})'(1)}{a\lambda_{0} + b} - \frac{a(y^{\#}_{1})'(1)}{\left( a\lambda_{0} + b \right)^{2}}
+\frac{a^2y'_{0}(1)}{\left( a\lambda_{0} + b\right)^{3}}, 
& \text{if }  \lambda_0 = -\dfrac{d}{c}.
\end{cases}\label{(5.5.3)}  
\end{equation}

\section{Equivalent forms of necessary and sufficient conditions}
The following three theorems prove the equivalence of the minimality conditions in \cite{aliyev5} and \cite{aliyev6}.
\begin{thm} 
If $\lambda_0$ is a double eigenvalue, then $\mathfrak{A}(y_{1}^{*}) = 0$ if and only if $C=-\frac{\omega^{\prime\prime\prime}(\lambda_0)}{3\omega^{\prime\prime}(\lambda_0)}$.
\end{thm}

\begin{proof}
Note that by (\ref{(5.1)}) and \(
\hat{y}_{1} = y_{1} + Cy_0\),
\begin{equation}
 C_1=-\frac{\mathfrak{A}({y}_{1})}{\mathfrak{A}(y_0)}-C-\frac{\omega^{\prime\prime\prime}(\lambda_0)}{3\omega^{\prime\prime}(\lambda_0)}. 
\end{equation}
By (\ref{(5.2.1)}) and \(y_{1}^* = y_{1} + C_1 y_0\), equality $\mathfrak{A}(y^{*}_{1})=0$ is equivalent to $\frac{\mathfrak{A}(y_{1})}{\mathfrak{A}(y_0)}+C_1=0$, which simplifies to
$C=-\frac{\omega^{\prime\prime\prime}(\lambda_0)}{3\omega^{\prime\prime}(\lambda_0)}$.
\end{proof}

\begin{thm} 
If $\lambda_0$ is a triple eigenvalue, then $\mathfrak{A}(y_{1}^{\#})= 0$ if and only if $C=-\frac{\omega^{IV}(\lambda_0)}{4\omega^{\prime\prime\prime}(\lambda_0)}$.
\end{thm}

\begin{proof}
By (\ref{(5.3)}) and \(
\hat{y}_{1} = y_{1} + Cy_0\),
\begin{equation}
 C_2=-\frac{\mathfrak{A}({y}_{1})}{\mathfrak{A}(y_0)}-C-\frac{\omega^{IV}(\lambda_0)}{4\omega^{\prime\prime\prime}(\lambda_0)}. \label{C2new}
\end{equation}
By (\ref{(5.5.1)}) and \(y_{1}^\# = y_{1} + C_2 y_0\), equality $\mathfrak{A}(y^{\#}_{1})=0$ is equivalent to $\frac{\mathfrak{A}(y_{1})}{\mathfrak{A}(y_0)}+C_2=0$, which simplifies to
$C=-\frac{\omega^{IV}(\lambda_0)}{4\omega^{\prime\prime\prime}(\lambda_0)}$.
\end{proof}

\begin{thm}
If $\lambda_0$ is a triple eigenvalue, then  $\mathfrak{A}(y_{2}^{\#})=0$ if and only if
\begin{equation}
 D=C^2+
\frac{\omega^{IV}(\lambda_0)}{4\omega^{\prime\prime\prime}(\lambda_0)}\left(C+\frac{\omega^{IV}(\lambda_0)}{4\omega^{\prime\prime\prime}(\lambda_0)}\right)-\frac{\omega^{V}(\lambda_0)}{20\omega^{\prime\prime\prime}(\lambda_0)}. \label{(6.25)}  
\end{equation}
\end{thm}

\begin{proof}
By (\ref{(5.9)}),
$$
D_1=-\frac{\mathfrak{A}(\hat{y}_{2})}{\mathfrak{A}(y_0)}--\frac{\mathfrak{A}(\hat{y}_{1})}{\mathfrak{A}(y_0)}\cdot\frac{\omega^{IV}(\lambda_0)}{4\omega^{\prime\prime\prime}(\lambda_0)}-\frac{\omega^{V}(\lambda_0)}{20\omega^{\prime\prime\prime}(\lambda_0)}+C_2^2.
$$
By using \(
\hat{y}_{1} = y_{1} + Cy_0,
\) and $\hat{y}_{2}=y_2+Cy_1+Dy_0$, we obtain
\begin{equation*}
    D_1=-\frac{\mathfrak{A}({y}_{2})+C\mathfrak{A}({y}_{1})+D\mathfrak{A}({y}_{0})}{\mathfrak{A}(y_0)}-\frac{\mathfrak{A}({y}_{1})+C\mathfrak{A}({y}_{0})}{\mathfrak{A}(y_0)}\cdot\frac{\omega^{IV}(\lambda_0)}{4\omega^{\prime\prime\prime}(\lambda_0)}-\frac{\omega^{V}(\lambda_0)}{20\omega^{\prime\prime\prime}(\lambda_0)}+C_2^2
\end{equation*}
\begin{equation}
    =-\frac{\mathfrak{A}({y}_{2})}{\mathfrak{A}(y_0)}-\frac{\mathfrak{A}({y}_{1})}{\mathfrak{A}(y_0)}\cdot\left(C+\frac{\omega^{IV}(\lambda_0)}{4\omega^{\prime\prime\prime}(\lambda_0)}\right)-\left(D+C\cdot\frac{\omega^{IV}(\lambda_0)}{4\omega^{\prime\prime\prime}(\lambda_0)}+\frac{\omega^{V}(\lambda_0)}{20\omega^{\prime\prime\prime}(\lambda_0)}\right)+C_2^2. \label{D1new}
\end{equation}
By  (\ref{(5.5.3)}), $\mathfrak{A}(y^{\#}_{2})=0$ is equivalent to $\frac{\mathfrak{A}(y_{2})}{\mathfrak{A}(y_0)}+C_2\frac{\mathfrak{A}(y_{1})}{\mathfrak{A}(y_0)}+D_1=0.$
By (\ref{C2new}) and (\ref{D1new}), this is transformed to
$$\frac{\mathfrak{A}(y_{2})}{\mathfrak{A}(y_0)}+C_2\cdot\left(-C_2-C-\frac{\omega^{IV}(\lambda_0)}{4\omega^{\prime\prime\prime}(\lambda_0)}\right)-\frac{\mathfrak{A}({y}_{2})}{\mathfrak{A}(y_0)}$$
$$-\left(-C_2-C-\frac{\omega^{IV}(\lambda_0)}{4\omega^{\prime\prime\prime}(\lambda_0)}\right)\cdot\left(C+\frac{\omega^{IV}(\lambda_0)}{4\omega^{\prime\prime\prime}(\lambda_0)}\right)-\left(D+C\cdot\frac{\omega^{IV}(\lambda_0)}{4\omega^{\prime\prime\prime}(\lambda_0)}+\frac{\omega^{V}(\lambda_0)}{20\omega^{\prime\prime\prime}(\lambda_0)}\right)+C_2^2=0.
$$
After cancelling $\frac{\mathfrak{A}(y_{2})}{\mathfrak{A}(y_0)}$, $C_2^2$ and $C_2\cdot\left(C+\frac{\omega^{IV}(\lambda_0)}{4\omega^{\prime\prime\prime}(\lambda_0)}\right)$, we obtain
$$
D=\left(C+\frac{\omega^{IV}(\lambda_0)}{4\omega^{\prime\prime\prime}(\lambda_0)}\right)^2-C\cdot \frac{\omega^{IV}(\lambda_0)}{4\omega^{\prime\prime\prime}(\lambda_0)}-\frac{\omega^{V}(\lambda_0)}{20\omega^{\prime\prime\prime}(\lambda_0)},
$$
which is equivalent to (\ref{(6.25)}).
\end{proof}
\section{Examples.} The following 3 examples are about the double eigenvalue case.
\subsection*{Example 1.}
Let us take the problem  
$$
-y^{\prime \prime }=\lambda y,\ 0<x<1,
$$
$$
y(0)=0,\ 
\left(\left(\frac{1}{\pi^2}+\frac{1}{3}\right)\lambda-1\right)y(1)=\left(\frac{1}{\pi^2}{\lambda }-1\right)y^{\prime }(1).
$$
Note that $a=\frac{1}{\pi^2}+\frac{1}{3}$, $b=-1$, $c=\frac{1}{\pi^2}$, $d=-1$, $ad-bc=-\frac{1}{3}<0$, $q(x)\equiv0$. We obtain that
$\lambda _{0}=\lambda _{1}=0$ is the double eigenvalue, and $\lambda _{0}\neq-\frac{d}{c}$. The second eigenvalue $\lambda _{2}={\pi^2}$ is simple and $\lambda _{2}=-\frac{d}{c}$. The remaining eigenvalues are ${\pi^2}<\lambda _{3}<\lambda _{4}<\ldots$. The eigenfunctions are $y_{0}=x$, $y_{2}=\frac{\sin{\pi x}}{\pi}$, $y_{n}=\frac{\sin{\sqrt{\lambda _{n}}x}}{{\sqrt{\lambda _{n}}}}$ $(n\ge 3)$, and the first associated function of $y_{0}$ is $y_{1}=-\frac{1}{6}x^3+Cx$, where $C$ is
a constant.
By formula (\ref{(Tk)}),
$$T_0:=(y_{1},y_{0})-\frac{1}{3}\left( \frac{y_{1}(1)}{-1} - \frac{1}{\pi^2}\cdot\frac{y_{0}(1)}{{(-1)}^{2}}\right)\cdot \frac{y_{0}(1)}{-1}=\frac{1}{45}-\frac{1}{3 \mathrm{\pi}^{2}}.$$
By (\ref{(2.17.1)}),
$$Q_{0}=(y_{1},y_{1})-\frac{1}{3}\left(\frac{y_{1}(1)}{-1} - \frac{1}{\pi^2}\cdot\frac{y_{0}(1)}{(-1)^{2}}
\right)^2
$$
$$=\frac{-315+\left(42 C-5\right) \mathrm{\pi}^{4}+\left(-630 C+105\right) \mathrm{\pi}^{2}}{945 \mathrm{\pi}^{4}}.$$
Then $C_{1}=-\frac{Q_0}{T_0}=\frac{315+\left(-42 C+5\right) \pi^{4}+\left(630 C-105\right) \pi^{2}}{21 \pi^{2} \left(\pi^{2}-15\right)}$. Therefore, 
$$
y_{1}^{*}=y_{1}+C_{1}y_{0}=-\frac{x \left(-15+\left(\frac{x^{2}}{6}+C-\frac{5}{21}\right) \pi^{4}+\left(-\frac{5 x^{2}}{2}-15 C+5\right) \pi^{2}\right)}{\pi^{2} \left(\pi^{2}-15\right)}.
$$
By (\ref{(5.2.1)}),
$$
\mathfrak{A}(y^{*}_{k+1}) =\frac{y_{1}^{*}(1)}{-1} -\frac{1}{{\pi}^{2}}\cdot \frac{y_{0}(1)}{(-1)^{2}}=\frac{\left(14 C-1\right) \pi^{2}-210 C+21}{14 \pi^{2}-210}.
$$
Consequently, $\mathfrak{A}(y^{*}_{k+1}) =0$ if and only if $C=\frac{\pi^{2}-21}{14 \left(\pi^{2}-15\right)}$.
On the other hand the characteristic function is
$$\omega(\lambda)=\frac{\left(\left(\frac{1}{\pi^{2}}+\frac{1}{3}\right)\cdot \lambda -1\right)\cdot \sin\! \left(\sqrt{\lambda}\right)}{\sqrt{\lambda}}-\left(\frac{\lambda}{\pi^{2}}-1\right)\cdot \cos\! \left(\sqrt{\lambda}\right).$$
Therefore, $\omega^{\prime\prime}(0)=\frac{-2 \pi^{2}+30}{45 \pi^{2}}$, $\omega^{\prime\prime\prime}(0)=\frac{\pi^{2}-21}{105 \pi^{2}}$. By Theorem 3.1, $$C=-\frac{\omega^{\prime\prime\prime}(\lambda_k)}{3\omega^{\prime\prime}(\lambda_k)}=\frac{\pi^{2}-21}{14 \pi^{2}-210},$$which confirms the earlier result that was obtained using the first method.

\subsection*{Example 2.}
Let us now take the problem  
$$
-y^{\prime \prime }=\lambda y,\ 0<x<1,
$$
$$
y(0)=0,\ 
\left(\lambda-\frac{\pi^2}{4}\right)y(1)=\left(\left(-\frac{\pi^2}{12}+1\right){\lambda }-\frac{\pi^2}{4}\right)y^{\prime }(1).
$$
Note that $a=1$, $b=-\frac{\pi^2}{4}$, $c=-\frac{\pi^2}{12}+1$, $d=-\frac{\pi^2}{4}$, $ad-bc=-\frac{\pi^4}{48}<0$. We obtain that
$\lambda _{0}=\lambda _{1}=0$ is the double eigenvalue, the second eigenvalue $\lambda _{2}=\frac{\pi^2}{4}$ is simple, and $\lambda _{0},\lambda _{2}\neq-\frac{d}{c}$. The remaining eigenvalues are $\frac{\pi^2}{4}<\lambda _{3}<\lambda _{4}<\ldots$. The eigenfunctions are $y_{0}=x$, $y_{2}=\frac{2\sin{\frac{\pi x}{2}}}{\pi}$, $y_{n}=\frac{\sin{\sqrt{\lambda _{n}}x}}{{\sqrt{\lambda _{n}}}}$ $(n\ge 3)$, and the first associated function of $y_{0}$ is again $y_{1}=-\frac{1}{6}x^3+Cx$, where $C$ is
a constant.
By formula (\ref{(Tk)}), $T_0=\frac{2 \pi^{2}-20}{15 \pi^{2}}.$
By (\ref{(2.17.1)}), $Q_{0}=\frac{-1680+\left(84 C-25\right) \pi^{4}+\left(-840 C+420\right) \pi^{2}}{315 \pi^{4}}$. Then $C_1=\frac{1680+\left(-84 C+25\right) \pi^{4}+\left(840 C-420\right) \pi^{2}}{42 \pi^{2} \left(\pi^{2}-10\right)}$, and therefore, 
$$
y_{1}^{*}=-\frac{x \left(-40+\left(\frac{x^{2}}{6}+C-\frac{25}{42}\right) \pi^{4}+\left(-\frac{5 x^{2}}{3}-10 C+10\right) \pi^{2}\right)}{\pi^{2} \left(\pi^{2}-10\right)}.
$$
By (\ref{(5.2.1)}), $
\mathfrak{A}(y^{*}_{k+1}) =\frac{\left(84 C-8\right) \pi^{2}-840 C+84}{21 \pi^{2} \left(\pi^{2}-10\right)}.
$
Consequently, $\mathfrak{A}(y^{*}_{k+1}) =0$ if and only if $C=\frac{2 \pi^{2}-21}{21 \left(\pi^{2}-10\right)}$.
If we approach the problem through the characteristic function
$$\omega(\lambda)=\frac{\left(\lambda -\frac{\pi^{2}}{4}\right)\cdot \sin\! \left(\sqrt{\lambda}\right)}{\sqrt{\lambda}}-\left(\left(-\frac{\pi^{2}}{12}+1\right)\cdot \lambda -\frac{\pi^{2}}{4}\right)\cdot \cos\! \left(\sqrt{\lambda}\right),$$
then $\omega^{\prime\prime}(0)=-\frac{\pi^{2}}{15}+\frac{2}{3}$, $\omega^{\prime\prime\prime}(0)=\frac{2 \pi^{2}}{105}-\frac{1}{5}$. By Theorem 3.1, $$C=-\frac{\omega^{\prime\prime\prime}(\lambda_k)}{3\omega^{\prime\prime}(\lambda_k)}=\frac{2 \pi^{2}-21}{21 \pi^{2}-210},$$which coincides with the above result.

\subsection*{Example 3.}
Let us now take the problem  
$$
-y^{\prime \prime }=\lambda y,\ 0<x<1,
$$
$$
y(0)=-y'(0),\ 
\left(-\left(\frac{3}{\pi^2}+1\right)\lambda+3\right)y(1)=\lambda y^{\prime }(1).
$$
Note that $a=-\left(\frac{3}{\pi^2}+1\right)$, $b=3$, $c=1$, $d=0$, $ad-bc=-3<0$. We obtain that
$\lambda _{0}=\lambda _{1}=0$ is the double eigenvalue, the second eigenvalue $\lambda _{2}={\pi^2}$ is simple, and $\lambda _{0}=-\frac{d}{c}\neq\lambda _{2}$. The remaining eigenvalues are $\pi^2<\lambda _{3}<\lambda _{4}<\ldots$. The eigenfunctions are $y_{0}=1-x$, $y_{2}=\cos{{\pi x}}-\frac{\sin{{\pi x}}}{\pi}$, $y_{n}=\cos{\sqrt{\lambda _{n}}x}-\frac{\sin{\sqrt{\lambda _{n}}x}}{{\sqrt{\lambda _{n}}}}$ $(n\ge 3)$, and the first associated function of $y_{0}$ is $y_{1}=\frac{1}{6}x^3-\frac{1}{2}x^2+C(1-x)$, where $C$ is
a constant.

By formula (\ref{(Tk)}), $T_0=-\frac{14}{45}-\frac{1}{3 \pi^{2}}.$
By (\ref{(2.17.1)}), $$Q_{0}=\frac{-315+\left(-588 C-194\right) \pi^{4}+\left(-630 C-525\right) \pi^{2}}{945 \pi^{4}}.$$ 
Then $C_1=\frac{-315+\left(-588 C-194\right) \pi^{4}+\left(-630 C-525\right) \pi^{2}}{294 \pi^{4}+315 \pi^{2}}$, and therefore, 
$$
y_{1}^{*}=\frac{x^{2} \left(-3+x\right)}{6}-C \left(-1+x\right)
$$
$$+\frac{\left(588 C \pi^{4}+194 \pi^{4}+630 C \pi^{2}+525 \pi^{2}+315\right) \left(-1+x\right)}{294 \pi^{4}+315 \pi^{2}}.
$$
By (\ref{(5.2.1)}), $
\mathfrak{A}(y^{*}_{k+1}) =\frac{\left(196 C-34\right) \pi^{2}+210 C-21}{588 \pi^{2}+630}.
$
Consequently, $\mathfrak{A}(y^{*}_{k+1}) =0$ if and only if $C=\frac{34 \pi^{2}+21}{14 \left(14 \pi^{2}+15\right)}$.
If we use the characteristic function
$$\omega(\lambda)=\left(-\left(\frac{3}{\pi^{2}}+1\right)\cdot \lambda +3\right)\cdot \left(\cos\! \left(\sqrt{\lambda}\right)-\frac{\sin\! \left(\sqrt{\lambda}\right)}{\sqrt{\lambda}}\right)
$$
$$+\lambda \cdot \left(\sqrt{\lambda}\cdot \sin\! \left(\sqrt{\lambda}\right)+\cos\! \left(\sqrt{\lambda}\right)\right),$$
then $\omega^{\prime\prime}(0)=\frac{28 \pi^{2}+30}{15 \pi^{2}}$, $\omega^{\prime\prime\prime}(0)=\frac{-34 \pi^{2}-21}{35 \pi^{2}}$. By Theorem 3.1, $$C=-\frac{\omega^{\prime\prime\prime}(\lambda_k)}{3\omega^{\prime\prime}(\lambda_k)}=\frac{34 \pi^{2}+21}{196 \pi^{2}+210},$$which confirms the above result.

We provide 2 more examples but for the triple eigenvalue case.

\subsection*{Example 4.}Consider the problem  
$$
-y^{\prime \prime }=\lambda y,\ 0<x<1,
$$
$$
y(0)=0,\ 
(3{\lambda}+\pi^2)y(1)= 2({\lambda }-\pi^2)y^{\prime }(1).
$$
The solution of problem $-y^{\prime \prime }=\lambda y$ with boundary conditions $y(0)=0$, $y^\prime(0)=1$ is $y({x,\lambda)}=\frac{1}{\sqrt{\lambda}}\sin{\sqrt{\lambda}x}$. Therefore, $y^\prime({x,\lambda)}=\cos{\sqrt{\lambda}x}$, and
the characteristic function is $$\omega(\lambda) = (3{\lambda}+\pi^2)\frac{\sin{\sqrt{\lambda}}}{{\sqrt{\lambda}}} -2({\lambda }-\pi^2)\cos{\sqrt{\lambda}}
$$
$$
=-\frac{(3+4\pi^{2})(\lambda-\pi^{2})^{3}}{24\pi^{4}}
+\frac{(15+8\pi^{2})(\lambda-\pi^{2})^{4}}{96\pi^{6}}
$$
$$
-\frac{(315+90\pi^{2}-8\pi^{4})(\lambda-\pi^{2})^{5}}{1920\pi^{8}}
+O\!\left((\lambda-\pi^{2})^{6}\right).
$$
Consequently,
$\omega(\pi^2)=\omega^\prime(\pi^2)=\omega^{\prime\prime}(\pi^2)=0$, $\omega^{\prime\prime\prime}(\pi^2)=\frac{-4 \pi^{2}-3}{4 \pi^{4}}$, $\omega^{IV}(\pi^2)=\frac{8 \pi^{2}+15}{4 \pi^{6}}$, and $\omega^{V}(\pi^2)=\frac{8 \pi^{4}-90 \pi^{2}-315}{16 \pi^{8}}$.
Finally, $\lambda _{0}=\lambda _{1}=\lambda _{2}=\pi^2$ is the triple eigenvalue such that $\lambda _{0}=-\frac{d}{c}=\pi^2$. All other eigenvalues $\lambda _{3}<\lambda _{4}<\ldots$ are real and simple.
The eigenfunctions are $y_{0}=\frac{\sin{\pi x}}{\pi}$ and $y_{n}=\frac{\sin{\sqrt{\lambda_n}x}}{{\sqrt{\lambda_n}}}$ $(n\ge 3)$.
To find the first associated function of $y_{0}$ we need to calculate the limit
$$\tilde{y}_1=\lim_{\lambda\to \pi^2}y_\lambda(x,\lambda)=\frac{x \cos\! \left(\pi  x\right) \pi -\sin\! \left(\pi  x\right)}{2 \pi^{3}}.$$
Similarly, for the second associated function
$$\tilde{y}_2=\frac{1}{2}\lim_{\lambda\to \pi^2}y_{\lambda\lambda}(x,\lambda)=\frac{-\sin\! \left(\pi  x\right) x^{2} \pi^{2}-3 x \cos\! \left(\pi  x\right) \pi +3 \sin\! \left(\pi  x\right)}{8 \pi^{5}}.$$
Thus, the first and second associated functions are 
$$y_{1}=\frac{x \cos\! \left(\pi  x\right) \pi -\sin\! \left(\pi  x\right)}{2 \pi^{3}}+C\cdot\frac{\sin{\pi x}}{\pi},$$and $$y_{2}=\frac{-\sin\! \left(\pi  x\right) x^{2} \pi^{2}-3 x \cos\! \left(\pi  x\right) \pi +3 \sin\! \left(\pi  x\right)}{8 \pi^{5}}
$$
$$+C\cdot\frac{x \cos\! \left(\pi  x\right) \pi -\sin\! \left(\pi  x\right)}{2 \pi^{3}}+D\cdot\frac{\sin{\pi x}}{\pi},$$where $C$ and $D$ are constants. 
In Theorem 3.2, $C=-\frac{\omega^{IV}(\lambda_k)}{4\omega^{\prime\prime\prime}(\lambda_k)}=\frac{8 \pi^{2}+15}{16 \pi^{4}+12 \pi^{2}}$. 
Similarly, in Theorem 3.3, \eqref{(6.25)} is equivalent to 
\begin{equation}
    \begin{aligned}
40\pi^{4}(4\pi^{2}+3)^{2}D={}&
90+640C^{2}\pi^{8}
+\left(960C^{2}-320C+16\right)\pi^{6} \\
&+\left(360C^{2}-840C-8\right)\pi^{4}
+\left(-450C-165\right)\pi^{2}.
\end{aligned} \label{D7.1}
\end{equation}

Let us now approach this example in a different way. By (\ref{(Qk)}), 
$Q_0=\frac{4 \pi^{2}+3}{96 \pi^{6}}$.
The same result is obtained if we  use (\ref{(2.17.1)}). By (\ref{(2.28)}),
${L_0}= \frac{-6+8 C \pi^{4}+\left(6 C-5\right) \pi^{2}}{96 \pi^{8}}$.
Then ${C_2}= -\frac{-6+8 C \pi^{4}+\left(6 C-5\right) \pi^{2}}{\pi^{2} \left(4 \pi^{2}+3\right)}$.
$$y_{1}^\# = y_{1} + C_2 y_0=\frac{\left(9-8 C \pi^{4}+\left(-6 C+6\right) \pi^{2}\right) \sin\! \left(\pi  x\right)+4 \left(\pi^{2}+\frac{3}{4}\right) x \pi  \cos\! \left(\pi  x\right)}{8 \pi^{5}+6 \pi^{3}}.$$
By (\ref{(5.5.1)}),  $ \mathfrak{A}(y^{\#}_{1})=\frac{-15+16 C \pi^{4}+\left(12 C-8\right) \pi^{2}}{64 \pi^{6}+48 \pi^{4}}.$ Therefore  $ \mathfrak{A}(y^{\#}_{1})=0$ if and only if  $C=\frac{8 \pi^{2}+15}{16 \pi^{4}+12 \pi^{2}}$, which coincides with the above result using the previous method.
Similarly, 
\[
\begin{aligned}
y^{*}_{2}
&=y_{2}+C_{2}y_{1}  \\
&=\frac{
\begin{aligned}
&\Bigl(
-15+(-64C^{2}+32D)\pi^{6}
+(-48C^{2}-4x^{2}+56C+24D)\pi^{4} \\
&\qquad
+(-3x^{2}+60C-8)\pi^{2}
\Bigr)\sin(\pi x) \\
&\qquad
-16\left(
-\frac{15}{16}
+C\pi^{4}
+\left(\frac{3C}{4}-\frac12\right)\pi^{2}
\right)x\pi\cos(\pi x)
\end{aligned}
}
{32\pi^{7}+24\pi^{5}}.
\end{aligned}
\]
By (\ref{(5.7)}), 

\[
J_0=
\frac{
\begin{aligned}
&-855
+\left(-3840C^{2}+2560D\right)\pi^{8}
+\left(-5760C^{2}+3200C+3840D-192\right)\pi^{6} \\
&\qquad
+\left(-2160C^{2}+6240C+1440D-704\right)\pi^{4}+\left(2880C-1020\right)\pi^{2}
\end{aligned}
}
{30720\pi^{12}+23040\pi^{10}}.
\]
Then 
\[
\begin{aligned}
D_1
&=-\frac{J_0}{Q_0} \\
&=\frac{
\begin{aligned}
&855+\left(3840C^{2}-2560D\right)\pi^{8}
+\left(5760C^{2}-3200C-3840D+192\right)\pi^{6}\\
&\qquad
+\left(2160C^{2}-6240C-1440D+704\right)\pi^{4}
+\left(-2880C+1020\right)\pi^{2}
\end{aligned}
}
{80\pi^{4}\left(4\pi^{2}+3\right)^{2}}.
\end{aligned}
\]
Consequently, 
\[
\begin{aligned}
y^{\#}_{2}
&=y^{*}_{2}+D_{1}y_{0} \\
&=\frac{
\begin{aligned}
&\sin(\pi x)\Bigl(
405+(1280C^{2}-1280D)\pi^{8} \\
&\qquad
+(1920C^{2}-160x^{2}-960C-1920D+192)\pi^{6} \\
&\qquad
+(720C^{2}-240x^{2}-2160C-720D+384)\pi^{4}\\
&\qquad
+(-90x^{2}-1080C+180)\pi^{2}
\Bigr) \\
&\qquad
-640x\left(
-\frac{15}{16}
+C\pi^{4}
+\left(\frac{3C}{4}-\frac12\right)\pi^{2}
\right) \\
&\qquad\qquad
{}\times\pi\left(\pi^{2}+\frac34\right)\cos(\pi x)
\end{aligned}
}
{80\pi^{5}(4\pi^{2}+3)^{2}}.
\end{aligned}
\]
By (\ref{(5.5.3)}),
\[
\mathfrak{A}(y^{\#}_{2})
=
\frac{
\begin{aligned}
&-90+\left(-640C^{2}+640D\right)\pi^{8}
+\left(-960C^{2}+320C+960D-16\right)\pi^{6} \\
&\qquad
+\left(-360C^{2}+840C+360D+8\right)\pi^{4}
+\left(450C+165\right)\pi^{2}
\end{aligned}
}
{160\pi^{6}\left(4\pi^{2}+3\right)^{2}}.
\]
Therefore, $ \mathfrak{A}(y^{\#}_{2})=0$ if and only if (\ref{D7.1}), which confirms the earlier result.

\subsection*{Example 5.}Consider the problem  
$$
-y^{\prime \prime }=\lambda y,\ 0<x<1,
$$
$$
y^\prime(0)=0,\ 
(12{\lambda}+\pi^2)y(1)= (8{\lambda }-2\pi^2)y^{\prime }(1).
$$
The solution of problem $-y^{\prime \prime }=\lambda y$ with boundary conditions $y(0)=1$, $y^\prime(0)=0$ is $y({x,\lambda)}=\cos{\sqrt{\lambda}x}$. Therefore, $y^\prime({x,\lambda)}=-{\sqrt{\lambda}}\sin{\sqrt{\lambda}x}$, and
the characteristic function is $$\omega(\lambda) = (12{\lambda}+\pi^2)\cos{\sqrt{\lambda}}+(8{\lambda }-2\pi^2){\sqrt{\lambda}}\sin{\sqrt{\lambda}}
$$
$$
=
-\frac{4(3+\pi^{2})}{3\pi^{3}}
\left(\lambda-\frac{\pi^{2}}{4}\right)^{3}
+\frac{12}{\pi^{5}}
\left(\lambda-\frac{\pi^{2}}{4}\right)^{4}
$$
$$+\frac{2(-270+15\pi^{2}+\pi^{4})}{15\pi^{7}}
\left(\lambda-\frac{\pi^{2}}{4}\right)^{5}
+O\!\left(\left(\lambda-\frac{\pi^{2}}{4}\right)^{6}\right).
$$
Consequently,
$\omega\left(\frac{\pi^2}{4}\right)=\omega^\prime\left(\frac{\pi^2}{4}\right)=\omega^{\prime\prime}\left(\frac{\pi^2}{4}\right)=0$, $\omega^{\prime\prime\prime}\left(\frac{\pi^2}{4}\right)=\frac{-8 \pi^{2}-24}{\pi^{3}}$, $\omega^{IV}\left(\frac{\pi^2}{4}\right)=\frac{288}{\pi^{5}}$, and $\omega^{V}\left(\frac{\pi^2}{4}\right)=\frac{16 \pi^{4}+240 \pi^{2}-4320}{\pi^{7}}$.
Finally, $\lambda _{0}=\lambda _{1}=\lambda _{2}=\frac{\pi^2}{4}$ is the triple eigenvalue such that $\lambda _{0}=-\frac{d}{c}=\frac{\pi^2}{4}$. All other eigenvalues $\lambda _{3}<\lambda _{4}<\ldots$ are real and simple.
The eigenfunctions are $y_{0}=\cos\frac{\pi}{2}x$ and $y_{n}=\cos{\sqrt{\lambda_n}x}$ $(n\ge 3)$.
To find the first associated function of $y_{0}$ we need to calculate the limit
$$\tilde{y}_1=\lim_{\lambda\to 0}y_\lambda(x,\lambda)=-\frac{ x}{\pi}\sin \! \frac{\pi  x}{2}.$$
Similarly, for the second associated function
$$\tilde{y}_2=\frac{1}{2}\lim_{\lambda\to 0}y_{\lambda\lambda}(x,\lambda)=-\frac{\left(x \cos\! \left(\frac{\pi  x}{2}\right) \pi -2 \sin\! \left(\frac{\pi  x}{2}\right)\right) x}{2\pi^{3}}.$$
Thus, the first and second associated functions are 
$$y_{1}=-\frac{ x}{\pi}\sin \! \frac{\pi  x}{2}+C\cdot\cos\frac{\pi}{2}x,$$and $$y_{2}=-\frac{\left(x \cos\! \left(\frac{\pi  x}{2}\right) \pi -2 \sin\! \left(\frac{\pi  x}{2}\right)\right) x}{2\pi^{3}}-C\cdot\frac{ x}{\pi}\sin \! \frac{\pi  x}{2}+D\cdot\cos\frac{\pi}{2}x,$$where $C$ and $D$ are constants. 
By Theorem 3.2,
$C=-\frac{\omega^{IV}(\lambda_k)}{4\omega^{\prime\prime\prime}(\lambda_k)}=\frac{9}{\pi^{2} \left(\pi^{2}+3\right)}$. 
Similarly, in Theorem 3.3,
 \eqref{(6.25)} is equivalent to 
\begin{equation}
 D=\frac{10 C^{2} \pi^{6}+\left(60 C^{2}+1\right) \pi^{4}+\left(90 C^{2}-90 C+18\right) \pi^{2}-270 C-225}{10 \pi^{2} \left(\pi^{2}+3\right)^{2}}. \label{D7}   
\end{equation}
Now we will show an alternative approach to this example. By (\ref{(Qk)}), 
$Q_0=\frac{\pi^{2}+3}{6 \pi^{4}}$.
The same result is obtained if we  use (\ref{(2.17.1)}). By (\ref{(2.28)})
${L_0}= \frac{-12+2 C \pi^{4}+\left(6 C-1\right) \pi^{2}}{6 \pi^{6}}$.
Then ${C_2}=-\frac{L_0}{Q_0}=-\frac{-12+2 C \pi^{4}+\left(6 C-1\right) \pi^{2}}{\pi^{2} \left(\pi^{2}+3\right)}$.
$$y_{1}^\# = y_{1} + C_2 y_0=\frac{\left(12-C \pi^{4}+\left(-3 C+1\right) \pi^{2}\right) \cos\! \left(\frac{\pi  x}{2}\right)-\sin\! \left(\frac{\pi  x}{2}\right) x \pi  \left(\pi^{2}+3\right)}{\pi^{2} \left(\pi^{2}+3\right)}.$$
By (\ref{(5.5.1)}),  $ \mathfrak{A}(y^{\#}_{1})=\frac{C \pi^{4}+3 C \pi^{2}-9}{8 \pi^{3} \left(\pi^{2}+3\right)}.$ Therefore  $ \mathfrak{A}(y^{\#}_{1})=0$ if and only if  $C=\frac{9}{\pi^{2} \left(\pi^{2}+3\right)}$, which coincides with the above result using the previous method.
Similarly,

\[
\begin{aligned}
y_{2}^{*}
&= y_{2}+C_{2}y_{1} \\
&=
\frac{-2\pi}{\pi^{3}(\pi^{2}+3)}
\begin{aligned}[t]
\Bigl(&\left(C^{2}-\frac{D}{2}\right)\pi^{4}
+\left(3C^{2}+\frac{x^{2}}{4}-\frac{C}{2}-\frac{3D}{2}\right)\pi^{2} \\
&\qquad
+\frac{3x^{2}}{4}-6C\Bigr)
\cos\!\left(\frac{\pi x}{2}\right)\\
&\qquad
+\frac{\sin\! \left(\frac{\pi  x}{2}\right) x \left(C \pi^{4}+3 C \pi^{2}-9\right)}{\pi^{3} \left(\pi^{2}+3\right)}.
\end{aligned}
\end{aligned}
\]
By (\ref{(5.7)}), 
$${J_0}= \frac{-135+\left(-15 C^{2}+10 \mathit{D}\right) \pi^{8}+\left(-90 C^{2}+10 C+60 \mathit{D}-3\right) \pi^{6}}{30 \pi^{8} \left(\pi^{2}+3\right)}$$

$$+\frac{\left(-135 C^{2}+150 C+90 \mathit{D}-24\right) \pi^{4}+\left(360 C+45\right) \pi^{2}}{30 \pi^{8} \left(\pi^{2}+3\right)}.$$
Then 
$$D_1=-\frac{J_0}{Q_0}=\frac{135+\left(15 C^{2}-10 \mathit{D}\right) \pi^{8}+\left(90 C^{2}-10 C-60 \mathit{D}+3\right) \pi^{6}}{5 \pi^{4} \left(\pi^{2}+3\right)^{2}}$$

$$+\frac{\left(135 C^{2}-150 C-90 \mathit{D}+24\right) \pi^{4}+\left(-360 C-45\right) \pi^{2}}{5 \pi^{4} \left(\pi^{2}+3\right)^{2}}.$$
Consequently, 
\[y^{\#}_{2}=y^{*}_{2}+D_{1}y_{0}=\frac{10 \pi  \sin\! \left(\frac{\pi  x}{2}\right) x  \left(C \pi^{4}+3 C \pi^{2}-9\right)}{10 \pi^{4} \left(\pi^{2}+3\right)}\]

\[
+\frac{
\begin{aligned}
\Bigl(&270+(10C^{2}-10D)\pi^{8}
+(60C^{2}-5x^{2}-10C-60D+6)\pi^{6}\\
&+(90C^{2}-30x^{2}-150C-90D+48)\pi^{4}
+(-45x^{2}-360C-90)\pi^{2}\Bigr)
\cos\!\left(\frac{\pi x}{2}\right)
\end{aligned}
}
{10\pi^{4}(\pi^{2}+3)^{2}}
.\]
By (\ref{(5.5.3)}),
$$\mathfrak{A}(y^{\#}_{2})=\frac{\left(-10 C^{2}+10 \mathit{D}\right) \pi^{6}+\left(-60 C^{2}+60 \mathit{D}-1\right) \pi^{4}}{80 \pi^{3} \left(\pi^{2}+3\right)^{2}}$$

$$+\frac{\left(-90 C^{2}+90 C+90 \mathit{D}-18\right) \pi^{2}+270 C+225}{80 \pi^{3} \left(\pi^{2}+3\right)^{2}}.$$
Therefore  $ \mathfrak{A}(y^{\#}_{2})=0$ if and only if (\ref{D7}), confirming the earlier result.

\subsection*{Acknowledgment}
This work was financially supported by ADA University and Baku State University. The authors would like to thank for the license of Maple software provided as a part of Maple Ambassador program.

\section*{Declarations}

\noindent\textbf{Ethics approval and consent to participate.}
Not applicable.

\noindent\textbf{Ethics declaration} Not applicable.

\noindent\textbf{Ethics, Consent to Participate, and Consent to Publish declarations.} Not applicable.

\noindent\textbf{Consent for publication.}
Not applicable.

\noindent\textbf{Competing interests.}
The authors declare that they have no competing interests.

\noindent\textbf{Authors' contributions.}
Both authors contributed equally to the conception, analysis, and writing of the manuscript. Both authors read and approved the final manuscript.

\noindent\textbf{Funding.}
This work was supported by the ADA University Faculty Research and Development Fund and Baku State University.

\noindent\textbf{Availability of data and materials.}
No datasets were generated or analysed during the current study. Data sharing is not applicable to this article.

\noindent\textbf{Acknowledgements.}
Not applicable.


\end{document}